\documentclass[12pt,reqno]{amsart}
\usepackage{mathrsfs,amsfonts,amsthm,amsmath,amssymb,thmtools}
\usepackage{hyperref}
\usepackage{cleveref}
\usepackage{enumitem}

\hypersetup{
  colorlinks=true, 
  urlcolor=blue, 
  linkcolor=blue, 
  citecolor=blue 
}

\usepackage{lineno}
 \theoremstyle{plain}

\theoremstyle{plain}

\newtheorem{theorem}{Theorem}[section]
\newtheorem{lemma}[theorem]{Lemma}
\newtheorem{proposition}[theorem]{Proposition}
\newtheorem{corollary}[theorem]{Corollary}
\newtheorem{question}[theorem]{Question}

\theoremstyle{definition}
\newtheorem{definition}[theorem]{Definition}
\newtheorem{remark}[theorem]{Remark}
\newtheorem{example}[theorem]{Example}

\makeatletter
\newcommand{\LeftEqNo}{\let\veqno\@@leqno}

\makeatother

\numberwithin{equation}  {section}

\begin{document}

\vspace{-2cm}

\title[Ortho-isomorphisms of von Neumann algebras]{Ortho-isomorphisms of von Neumann algebras}

\author{Minghui Ma}
\address{Minghui Ma, School of Mathematical Sciences, Dalian University of Technology, Dalian, 116024, China}
\email{minghuima@dlut.edu.cn}

\author{Weijuan Shi}
\address{Weijuan Shi, School of Mathematics and Statistics, Shaanxi Normal University, Xi'an, 710062, China}
\email{shiweijuan1016@163.com}

\thanks{Minghui Ma was supported by the Postdoctoral Fellowship Program of CPSF (No. GZC20252022).
Weijuan Shi was supported by the National Natural Science Foundation of China (No. 12571139), the Natural Science Basic Research Plan in Shaanxi Province (No. 2023-JC-YB-050), Shaanxi Fundamental Science Research Project for Mathematics and Physics (No. 23JSQ038), and the Fundamental Research Funds for the Central Universities (No. GK202501008).}

\begin{abstract}
Suppose $\mathscr M$ and $\mathscr N$ are von Neumann algebras.
Two operators $A$ and $B$ in $\mathscr M$ are said to be orthogonal if $A^*B=0$, meaning their ranges are orthogonal.
Let $\varphi\colon\mathscr M\to\mathscr N$ be a map.
We say that $\varphi$ is an ortho-isomorphism if it is bijective and satisfies that $A^*B=0$ if and only if $\varphi(A)^*\varphi(B)=0$ for all $A,B\in\mathscr M$.
The map $\varphi$ is called ortho-additive if the additive relation $\varphi(A+B)=\varphi(A)+\varphi(B)$ holds for all $A,B\in \mathscr M$ with $A^*B=0$.
In this paper, we characterize the complete structure of ortho-additive ortho-isomorphisms between von Neumann algebras, which is an analogue of Dye's theorem and Uhlhorn's theorem.
\end{abstract}


\maketitle
\label{}
\baselineskip18pt

\section{Introduction}

There is a long history of investigating preserver problems concerning various relations and structures on operator algebras, and many remarkable results have been established over the past decades \cite{Dye,Kad51,Uhl,Wig}.
The celebrated Wigner's theorem \cite{Wig} states that every symmetry transformation on the ray space $\mathscr P_1(\mathscr H)$, the set of all rank-one projections on a complex Hilbert space $\mathscr H$, is induced by a unitary or an anti-unitary operator.
This result has been generalized by many authors \cite{BJM,BMS,Geh,Gyo,Mol01,Uhl,WWY}.
Among these studies, Uhlhorn \cite{Uhl} provided a pivotal generalization by characterizing the structure of transformations preserving orthogonality on $\mathscr P_1(\mathscr H)$ for $\dim\mathscr H\geqslant 3$.
Since projection lattices play an important role in the theory of von Neumann algebras, a lot of research work in this direction has been devoted to studying transformations and related problems between the projection lattices of von Neumann algebras.
Dye \cite{Dye} showed that every ortho-isomorphism between the projection lattices of von Neumann algebras without direct summand of type $\mathrm{I}_2$ can be extended to a Jordan *-isomorphism.
It is worth noting that, without assuming surjectivity, Qian et al. in \cite{QWWY} characterized the general form of transition probability preserving maps on $\mathscr P_c(\mathscr M)$, the set of all projections with a fixed trace $c$ in a semifinite factor $\mathscr M$, thus providing a non-surjective generalization of Wigner's theorem.
In \cite{Mor}, Mori investigated surjective isometries between the projection lattices of von Neumann algebras.
Afterwards, the authors in \cite{SSGM,SSM} gave a complete characterization
of ortho-preserving transformations on $\mathscr P_c(\mathscr M)$, thereby establishing a version of Uhlhorn's theorem in the setting of semifinite factors.
It is therefore natural to ask for a description of ortho-preserving transformations between the whole von Neumann algebras.
Notice that the orthogonality of operators is closely related to the *-product of operators.
In recent years, many scholars have studied transformations preserving certain algebraic relations on operator algebras, including the *-product, Lie product, *-Lie product, Jordan product, and so on \cite{CKLW,CH,KLW,LJ,Sem93,WJ}.

Throughout this paper, let $\mathscr M$ denote a von Neumann algebra.
Recall that the *-product of two operators $A$ and $B$ in $\mathscr M$ is defined as $A^*B$.
It is clear that $A$ and $B$ have zero *-product if and only if their range projections $\mathcal{R}(A)$ and $\mathcal{R}(B)$ are orthogonal.
For the sake of convenience, we adopt the following terminology:
two operators $A$ and $B$ in $\mathscr M$ are said to be orthogonal, written as $A\perp B$, if they have a zero *-product.
Let $\varphi\colon\mathscr M\to\mathscr N$ be a map between von Neumann algebras.
We say that $\varphi$ is ortho-additive if $\varphi(A+B)=\varphi(A)+\varphi(B)$ whenever $A\perp B$.
If $A\perp B$ implies that $\varphi(A)\perp\varphi(B)$, then $\varphi$ is called ortho-preserving.
Therefore, $\varphi$ being ortho-preserving is equivalent to saying that $\varphi$ preserves zero *-product.
In particular, $\varphi$ is called an ortho-isomorphism if it is bijective and ortho-preserving in both directions.
As a special case, if $\Phi\colon\mathscr M\to\mathscr M$ is a range-contractive map, i.e., $\mathcal{R}(\Phi(A))\leqslant\mathcal{R}(A)$, then $\Phi$ is ortho-preserving.
When considering transformations between operator algebras, linearity is usually imposed as an additional condition \cite{CH,Kad51,Sem08}.
In this paper, we will replace linearity with ortho-additivity, which is a much weaker condition.
We aim to describe the general form of ortho-additive ortho-isomorphisms between von Neumann algebras.
Using Dye's result \cite[Corollary of Theorem 1]{Dye}, we discover that there is a strong connection between ortho-isomorphisms and range-contractive maps.
Thus, we first establish the following result in \Cref{thm main1}.

\begin{theorem}\label{main1}
Suppose $\mathscr M$ is a von Neumann algebra without direct summand of type
$\mathrm{I}_1$ and $\Phi$ is an ortho-additive range-contractive map on $\mathscr{M}$.
Then $\Phi(A)=A\Phi(I)$ for every $A\in\mathscr M$.
\end{theorem}

As an application of \Cref{main1} and Dye's result, we further study ortho-additive ortho-isomorphisms between von Neumann algebras.
Since Dye's result applies only to von Neumann algebras without direct summand of type $\mathrm{I}_2$, the case of type $\mathrm{I}_2$ von Neumann algebras is considered separately in \Cref{prop type-I2}.
The following theorem (see \Cref{thm main2}) is the main result in this paper.

\begin{theorem}\label{main2}
Suppose $\mathscr M$ and $\mathscr N$ are von Neumann algebras such that $\mathscr M$ has no direct summand of type $\mathrm{I}_1$, and $\varphi\colon\mathscr M\to\mathscr N$ is an ortho-additive ortho-isomorphism.
Then there exists a map $\Phi\colon\mathscr M\to\mathscr N$ such that
\begin{equation*}
  \varphi(A)=\Phi(A)\varphi(I)\quad\text{for all}~A\in\mathscr M,
\end{equation*}
and $\Phi$ is the direct sum of a *-isomorphism and a conjugate *-isomorphism.
\end{theorem}

In both \Cref{main1} and \Cref{main2}, the assumption that $\mathscr M$ has no direct summand of type $\mathrm{I}_1$ is essential.
Otherwise, the next example demonstrates that the conclusions may fail.

\begin{example}\label{eg strange}
Let $f$ be a homeomorphism from $\mathbb{C}$ onto $\mathbb{C}$ such that $f(0)=0$.
For example, if we take $f(re^{i\theta})=r^2e^{i\theta}$ in polar coordinates, then $f$ is a homeomorphism, whose inverse is given by $f^{-1}(re^{i\theta})=\sqrt{r}e^{i\theta}$.
Let $\mathscr M$ be an abelian von Neumann algebra.
Then we can define an ortho-additive range-contractive map by
\begin{equation*}
  \Phi\colon\mathscr M\to\mathscr M,\quad \Phi(A)=f(A).
\end{equation*}
Moreover, $\Phi$ is an ortho-isomorphism.
In the special case $\mathscr M=\mathbb{C}$, the map $\lambda\mapsto f(\lambda)$ is an ortho-additive range-contractive ortho-isomorphism for any bijection $f$ on $\mathbb{C}$ with $f(0)=0$.
\end{example}

In a $C^*$-algebra, we can similarly define the orthogonal relation $A\perp B$ by $A^*B=0$.
Thus, it is reasonable to consider ortho-additive ortho-isomorphisms between $C^*$-algebras.
Recall that a UHF $C^*$-algebra is the inductive limit of a sequence of finite-dimensional full matrix algebras \cite[Proposition 10.4.18]{KR2}.
We would like to know whether a conclusion similar to \Cref{main2} holds for UHF $C^*$-algebras.

\begin{question}\label{que UHF}
Suppose $\mathscr M$ and $\mathscr N$ are UHF $C^*$-algebras, and $\varphi\colon\mathscr M\to\mathscr N$ is an ortho-additive ortho-isomorphism.
How to describe the structure of $\varphi$?
\end{question}

This paper is organized as follows.
In \Cref{sec preliminaries}, we provide some definitions, notation, and lemmas that will be used in later sections.
In \Cref{sec range-contractive}, we obtain the general form of ortho-additive range-contractive maps on von Neumann algebras in \Cref{thm main1}.
In \Cref{sec ortho-isomorphism}, we study ortho-additive ortho-isomorphisms between von Neumann algebras and prove our main result in \Cref{thm main2}.

\section{Preliminaries}\label{sec preliminaries}

Let $\mathscr H$ be a complex Hilbert space (not necessarily separable) and $\mathscr{B(H)}$ the algebra of all bounded linear operators on $\mathscr H$.
A {\em von Neumann algebra} is a self-adjoint unital subalgebra of $\mathscr{B(H)}$ that is closed in the weak-operator topology.
By the type decomposition theorem (see \cite[Theorem 6.5.2]{KR2}), every von Neumann algebra is uniquely decomposable into the direct sum of those of type $\mathrm{I}_n$, type $\mathrm{I}_\infty$, type $\mathrm{II}_1$, type $\mathrm{II}_\infty$, and type $\mathrm{III}$.
For more details on the theory of von Neumann algebras, the reader is referred to \cite{KR1,KR2}.

Throughout this paper, let $\mathscr M$ be a von Neumann algebra.
For every operator $A$ in $\mathscr M$, its {\em range projection} is denoted by $\mathcal{R}(A)$.
Two operators $A$ and $B$ in $\mathscr M$ are called {\em orthogonal}, denoted by $A\perp B$, if $A^*B=0$.
It is clear that $A$ and $B$ are orthogonal if and only if their range projections are orthogonal, i.e., $\mathcal{R}(A)\perp\mathcal{R}(B)$.

\begin{definition}\label{def ortho}
Suppose $\mathscr M$ and $\mathscr N$ are von Neumann algebras, and $\varphi\colon\mathscr M\to\mathscr N$ is a map.
\begin{enumerate}
\item [(i)] $\varphi$ is called {\em ortho-additive} if for all $A,B\in\mathscr M$ with $A\perp B$, we have
\begin{equation*}
  \varphi(A+B)=\varphi(A)+\varphi(B).
\end{equation*}

\item [(ii)] $\varphi$ is called {\em ortho-preserving} if for all $A,B\in\mathscr M$ with $A\perp B$, we have
\begin{equation*}
  \varphi(A)\perp\varphi(B).
\end{equation*}

\item [(iii)] $\varphi$ is called an {\em ortho-isomorphism} if $\varphi$ is bijective and both $\varphi,\varphi^{-1}$ are ortho-preserving, i.e., for all $A,B\in\mathscr M$, we have
\begin{equation*}
    A\perp B\quad\text{if and only if}\quad\varphi(A)\perp\varphi(B).
\end{equation*}

\item [(iv)] A map $\Phi\colon\mathscr M\to\mathscr M$ is called {\em range-contractive} if $\mathcal{R}(\Phi(A))\leqslant\mathcal{R}(A)$ for all $A\in\mathscr M$.
\end{enumerate}
\end{definition}

The propose of this paper is to characterize the complete structure of ortho-additive ortho-isomorphisms between von Neumann algebras (see \Cref{thm main2}).

For every von Neumann algebra $\mathscr M$, its projection lattice $\mathscr P(\mathscr M)$ is the set of all projections in $\mathscr M$, i.e.,
\begin{equation*}
  \mathscr P(\mathscr M)=\{P\in \mathscr M\colon P^2=P=P^*\}.
\end{equation*}
Suppose $\mathscr M$ and $\mathscr N$ are von Neumann algebras.
A map $\theta\colon P(\mathscr M)\to\mathscr P(\mathscr N)$ is said to be {\em a projection ortho-isomorphism} it is bijective and preserves orthogonality in both directions, i.e., $P\perp Q$ if and only if $\theta(P)\perp\theta(Q)$ for all $P,Q\in\mathscr P(\mathscr M)$.
Let $\rho\colon\mathscr M\to\mathscr N$ be a map and $\rho_1(A)=\rho(A^*)$ for all $A\in\mathscr M$.
It is clear that $\rho$ is a *-anti-isomorphism if and only if $\rho_1$ is a conjugate *-isomorphism.
Note that $\rho_1(P)=\rho(P)$ for every $P\in\mathscr P(\mathscr M)$.
Therefore, the following theorem is a reformulation of \cite[Corollary of Theorem 1]{Dye}.

\begin{theorem}\label{thm Dye}
Suppose $\mathscr M$ and $\mathscr N$ are von Neumann algebras such that $\mathscr M$ has no direct summand of type $\mathrm{I}_2$, and $\theta\colon\mathscr P(\mathscr M)\to\mathscr P(\mathscr N)$ is a projection ortho-isomorphism.
Then $\theta$ can be extended to a map from $\mathscr M$ onto $\mathscr N$ which is the direct sum of a *-isomorphism and a conjugate *-isomorphism.
\end{theorem}

The following lemma is taken from \cite[Lemma 1]{Dye}.

\begin{lemma}\label{lem Dye}
Suppose $\mathscr M$ and $\mathscr N$ are von Neumann algebras, and $\theta\colon\mathscr P(\mathscr M)\to\mathscr P(\mathscr N)$ is a projection ortho-isomorphism.
Then $\theta$ preserves the following entities:
$0$, $I$; the ortho-complement $I-P$ of $P$; order; and commutativity.
\end{lemma}

Recall that a projection $P$ in $\mathscr M$ is said to be an {\em abelian projection} if $P\mathscr MP$ is abelian.
It is clear that $P$ is abelian if and only if $P_1P_2=P_2P_1$ for all projections $P_1,P_2\leqslant P$.
The {\em central support} $C_P$ of $P$ is the minimal central projection in $\mathscr M$ such that $P\leqslant C_P$.
By the type decomposition theorem, for every positive integer $n$, there exists a maximal central projection $Z_n$ in $\mathscr M$ such that $\mathscr MZ_n$ is of type $\mathrm{I}_n$ or $Z_n=0$.
By \cite[Proposition 6.4.6(iii)]{KR2}, $Z_n$ is the maximal central projection in $\mathscr M$ such that there are mutually orthogonal abelian projections $\{P_j\}_{j=1}^n$ satisfying that $Z_n=\sum_{j=1}^{n}P_j$ and $Z_n=C_{P_j}$ for every $1\leqslant j\leqslant n$.

\begin{corollary}\label{cor Dye}
Suppose $\mathscr M$ and $\mathscr N$ are von Neumann algebras, and $\theta\colon\mathscr P(\mathscr M)\to\mathscr P(\mathscr N)$ is a projection ortho-isomorphism.
Then $\theta$ preserves the following entities:
the set of central projections; the central support $C_P$ of $P$; the set of abelian projections; the maximal central projection $Z_n$ such that $\mathscr MZ_n$ is of type $\mathrm{I}_n$ or $Z_n=0$.
\end{corollary}

\begin{proof}
Since $\theta$ preserves commutativity by \Cref{lem Dye}, $\theta$ preserves the set of central projections.
Since $C_P$ is the minimal central projection such that $P\leqslant C_P$ and $\theta$ preserves order by \Cref{lem Dye}, $\theta$ preserves the central support $C_P$ of $P$, i.e., $\theta(C_P)=C_{\theta(P)}$.
Since $P$ is abelian if and only if $P_1P_2=P_2P_1$ for all projections $P_1,P_2\leqslant P$, we know that $\theta$ preserves the set of abelian projections by \Cref{lem Dye}.
If $\{P_j\}_{j=1}^n$ is a family of mutually orthogonal projections, then $\sum_{j=1}^{n}P_j$ is the minimal projection $P$ such that $P_j\leqslant P$ for every $1\leqslant j\leqslant n$.
It follows that $\theta(\sum_{j=1}^{n}P_j)=\sum_{j=1}^{n}\theta(P_j)$.
Since $Z_n$ is the maximal central projection such that there are mutually orthogonal abelian projections $\{P_j\}_{j=1}^n$ satisfying that $Z_n=\sum_{j=1}^{n}P_j$ and $Z_n=C_{P_j}$ for every $1\leqslant j\leqslant n$, the projection $\theta(Z_n)$ has the same property.
Hence $\theta(Z_n)$ is the maximal central projection such that $\mathscr N\theta(Z_n)$ is of type $\mathrm{I}_n$ or $\theta(Z_n)=0$.
\end{proof}

The following technique lemma will be used in next section.
Let $V$ be a partial isometry with initial projection $P$ and final projection $Q$.
It is clear that $PQ=0$ if and only if $V^2=0$.

\begin{lemma}\label{lem amazing}
Suppose $\mathscr M$ is a von Neumann algebra and $V$ is a partial isometry in $\mathscr M$ with orthogonal initial projection $P$ and final projection $Q$, i.e., $P=V^*V$, $Q=VV^*$, and $PQ=0$.
Let $S$ be a positive operator in $P\mathscr MP$ such that $0\leqslant S\leqslant(1-\varepsilon)P$ for some positive number $\varepsilon>0$, $C=\sqrt{P-S^2}$, and $R=C^2+CSV^*+VCS+VS^2V^*$ a projection in $\mathscr M$.
If $A$ and $B$ are operators in $\mathscr M$ such that
\begin{equation*}
  PA=A,\quad QB=B,\quad\text{and}\quad R(A+B)=A+B,
\end{equation*}
then $B=V(C^{-1}S)A$, where $C^{-1}$ is the inverse of $C$ in $P\mathscr MP$.
\end{lemma}

\begin{proof}
Since $PR=C^2+CSV^*$ and $QR=VCS+VS^2V^*$, we have
\begin{align*}
  A & =P(A+B)=PR(A+B)=C(C+SV^*)(A+B), \\
  B & =Q(A+B)=QR(A+B)=VS(C+SV^*)(A+B).
\end{align*}
It follows that $CV^*B=CS(C+SV^*)(A+B)=SA$.
Since $0\leqslant S\leqslant(1-\varepsilon)P$ for some $\varepsilon>0$ and $C=\sqrt{P-S^2}$, the operator $C$ is invertible in $P\mathscr MP$, whose inverse is denoted by $C^{-1}$.
Thus, $V^*B=(C^{-1}S)A$ and
\begin{equation*}
  B=QB=VV^*B=V(C^{-1}S)A.
\end{equation*}
This completes the proof.
\end{proof}

Suppose $\mathscr M$ is a von Neumann algebra acting on a complex Hilbert space $\mathscr H$.
Recall that a closed densely defined operator $A$ on $\mathscr H$ is {\em affiliated} with $\mathscr M$ if $U^*AU=A$ for each unitary operator $U$ commuting with $\mathscr M$.
Let $\eta\mathscr M$ be the set of all (unbounded) operators affiliated with $\mathscr M$.
If $\mathscr M$ is an abelian von Neumann algebra, then $\eta\mathscr M$ forms a commutative *-algebra by \cite[Theorem 5.6.15]{KR1}.
In general, by using center-valued trace, $\eta\mathscr M$ forms a *-algebra if $\mathscr M$ is a finite von Neumann algebra.
If $\mathscr N$ is a finite von Neumann algebra and $\mathscr M=M_n(\mathscr N)$, then $\mathscr M$ is a finite von Neumann algebra and $\eta\mathscr M=M_n(\eta\mathscr N)$.
The following lemma describes the structure of unital *-ring homomorphisms from $\mathscr A$ into $\eta\mathscr B$ for abelian von Neumann algebras.

\begin{lemma}\label{lem *ring-homomorphism}
Suppose $\mathscr A$ and $\mathscr B$ are abelian von Neumann algebras and $f\colon\mathscr A\to\eta\mathscr B$ is a unital *-ring homomorphism, i.e.,
\begin{equation*}
  f(1)=1,\quad f(a^*)=f(a)^*,\quad f(a+b)=f(a)+f(b),\quad\text{and}\quad f(ab)=f(a)f(b).
\end{equation*}
Then $f$ is a map from $\mathscr A$ into $\mathscr B$ and is the direct sum of a *-homomorphism and a conjugate *-homomorphism.
\end{lemma}

\begin{proof}
Since $f(a+b)=f(a)+f(b)$, we know that $f(ra)=rf(a)$ for every rational number $r\in\mathbb{Q}$.
Since $f(a^*a)=f(a)^*f(a)\geqslant 0$ for every $a\in\mathscr A$, $f$ is a positive map.
For any rational number $r\geqslant\|a\|^2$, we have
\begin{equation*}
  f(a)^*f(a)=f(a^*a)\leqslant f(r1)=rf(1)=r1.
\end{equation*}
It follows that $\|f(a)\|\leqslant\|a\|$ and hence $f$ is a map from $\mathscr A$ into $\mathscr B$.
Moreover, for any $a,b\in\mathscr A$, we have
\begin{equation*}
  \|f(a)-f(b)\|=\|f(a-b)\|\leqslant\|a-b\|.
\end{equation*}
Thus, $f$ is a continuous map.
In particular, $f(ra)=rf(a)$ for every real number $r\in\mathbb{R}$.

Since $f(i1)^2=f((i1)^2)=f(-1)=-f(1)=-1$, there are projections $q_1$ and $q_2$ in $\mathscr B$ such that $1=q_1+q_2$ and $f(i1)=iq_1-iq_2$.
For $j=1,2$, let $\mathscr B_j=\mathscr Bq_j$ and
\begin{equation*}
  f_j\colon\mathscr A\to\mathscr B_j,\quad f_j(a)=f(a)q_j.
\end{equation*}
Then $\mathscr B=\mathscr B_1\oplus\mathscr B_2$ and $f=f_1\oplus f_2$.
It is clear that $f_1$ is a *-homomorphism and $f_2$ is a conjugate *-homomorphism.
This completes the proof.
\end{proof}

At the end of this section, we recall some relevant knowledge of von Neumann algebras that will be used in later sections.
Let $\mathscr M$ be a von Neumann algebra.
Two projections $P$ and $Q$ in $\mathscr M$ are said to be {\em (Murray-von Neumann) equivalent}, denoted by $P\sim Q$, when $V^*V=P$ and $VV^*=Q$ for some partial isometry $V$ in $\mathscr M$.
We say that $P$ is {\em weaker} than $Q$, written as $P\preceq Q$, if $P$ is equivalent to a subprojection of $Q$.
A projection $P$ is called {\em infinite} if $P\sim P_0<P$ for some projection $P_0$ in $\mathscr M$.
Otherwise, $P$ is said to be {\em finite}.
If $P$ is infinite and $PZ$ is either $0$ or infinite for each central projection $Z$ in $\mathscr M$, then $P$ is said to be {\em properly infinite}.
By the halving lemma (see \cite[Lemma 6.3.3]{KR2}), if $P$ is properly infinite, then there are projections $P_1$ and $P_2$ in $\mathscr M$ such that $P=P_1+P_2\sim P_1\sim P_2$.

\section{Ortho-additive range-contractive maps}\label{sec range-contractive}

In this section, let $\Phi$ be an ortho-additive range-contractive map on a von Neumann algebra $\mathscr M$ (see \Cref{def ortho}).
Due to \Cref{eg strange}, we determine the structure of $\Phi$ mainly under the assumption that $\mathscr M$ has no direct summand of type $\mathrm{I}_1$ in \Cref{thm main1}.
The following lemma is straightforward and will be used frequently.

\begin{lemma}\label{lem PA}
Let $\Phi$ be an ortho-additive range-contractive map on a von Neumann algebra $\mathscr{M}$.
Then for any projection $P$ and any operator $A$ in $\mathscr M$, we have $\Phi(PA)=P\Phi(A)$.
\end{lemma}

\begin{proof}
Since $PA\perp(I-P)A$, it follows from the ortho-additivity of $\Phi$ that
\begin{equation*}
  \Phi(A)=\Phi(PA)+\Phi((I-P)A).
\end{equation*}
Since $\Phi$ is range-contractive, we have
\begin{equation*}
  P\Phi(PA)=\Phi(PA)\quad\text{and}\quad (I-P)\Phi((I-P)A)=\Phi((I-P)A).
\end{equation*}
Thus, $P\Phi(A)=P\Phi(PA)=\Phi(PA)$.
\end{proof}

As an application of \Cref{lem amazing}, we can obtain the following lemma.

\begin{lemma}\label{lem VBA}
Let $\Phi$ be an ortho-additive range-contractive map on a von Neumann algebra $\mathscr{M}$.
Suppose $V$ is a partial isometry in $\mathscr M$ with orthogonal initial projection $P$ and final projection $Q$, i.e., $P=V^*V$, $Q=VV^*$, and $PQ=0$.
Then for any $A\in\mathscr M$ and $B\in P\mathscr MP$, we have
\begin{equation*}
  \Phi(VBA)=VB\Phi(A).
\end{equation*}
In particular, $\Phi(VA)=V\Phi(A)$ for every $A\in\mathscr M$.
\end{lemma}

\begin{proof}
Let $S$, $C$, and $R$ be operators given by \Cref{lem amazing}.
For any $A\in \mathscr M$, we have
\begin{equation*}
  RA=(C^2+CSV^*)A+(VCS+VS^2V^*)A.
\end{equation*}
Let $A_0=(C^2+CSV^*)A$ and $B_0=(VCS+VS^2V^*)A$.
Since $A_0\perp B_0$, it follows from the ortho-additivity of $\Phi$ that
\begin{equation*}
  \Phi(A_0+B_0)=\Phi(A_0)+\Phi(B_0).
\end{equation*}
Note that $PA_0=A_0$, $QB_0=B_0$, and $R(A_0+B_0)=A_0+B_0$.
It follows from \Cref{lem PA} that
\begin{equation*}
  P\Phi(A_0)=\Phi(A_0),\quad Q\Phi(B_0)=\Phi(B_0),\quad\text{and}\quad R(\Phi(A_0)+\Phi(B_0))=\Phi(A_0)+\Phi(B_0).
\end{equation*}
By \Cref{lem amazing}, we have $\Phi(B_0)=V(C^{-1}S)\Phi(A_0)$, i.e.,
\begin{equation}\label{equ B0}
  \Phi((VCS+VS^2V^*)A)=V(C^{-1}S)\Phi((C^2+CSV^*)A).
\end{equation}
A direct computation shows that
\begin{equation*}
  (VCS+VS^2V^*)(P+V(C^{-1}S))=V(C^{-1}S)\quad\text{and}\quad
  (C^2+CSV^*)(P+V(C^{-1}S))=P.
\end{equation*}
Substituting $A$ with $(P+V(C^{-1}S))A$ in \eqref{equ B0}, we obtain that
\begin{equation}\label{equ pre-VHA}
  \Phi(V(C^{-1}S)A)=V(C^{-1}S)\Phi(PA)\quad\text{for all}~A\in \mathscr M.
\end{equation}
We define a continuous function $f(t)=\frac{t}{\sqrt{t^2+1}}$ for $t\in[0,+\infty)$.
Let $H$ be a positive operator in $P\mathscr MP$ and $S=f(H)$.
It is routine to check that $C^{-1}S=H$.
Since the operator $S$ in $P\mathscr MP$ satisfies the condition in \Cref{lem amazing}, we have $\Phi(VHA)=VH\Phi(PA)$ by \eqref{equ pre-VHA}.
Combining with \Cref{lem PA}, we can get
\begin{equation}\label{equ VHA}
  \Phi(VHA)=VH\Phi(A)
\end{equation}
for all $A\in\mathscr M$ and all positive $H\in P\mathscr MP$.

Next, let $B$ be an operator in $P\mathscr MP$ and $B=UH$ the polar decomposition, where $U$ is a partial isometry and $H$ is a positive operator in $P\mathscr MP$.
Let
\begin{equation*}
  V_1=VU,\quad P_1=V_1^*V_1=U^*U\leqslant P,\quad\text{and}\quad Q_1=V_1V_1^*\leqslant Q.
\end{equation*}
Then $H$ is a positive operator in $P_1\mathscr MP_1$.
It follows from \eqref{equ VHA} that
\begin{equation*}
  \Phi(VBA)=\Phi(V_1HA)=V_1H\Phi(A)=VB\Phi(A).
\end{equation*}
This completes the proof.
\end{proof}

The following lemma is a direct consequence of \Cref{lem VBA}.

\begin{lemma}\label{lem BA}
Let $\Phi$ be an ortho-additive range-contractive map on a von Neumann algebra $\mathscr{M}$.
Suppose $P$ is a projection in $\mathscr M$ such that $P\preceq I-P$.
Then for any $A\in\mathscr M$ and $B\in P\mathscr MP$, we have
\begin{equation*}
  \Phi(BA)=B\Phi(A).
\end{equation*}
In particular, $\Phi(B)=B\Phi(I)$.
\end{lemma}

\begin{proof}
Since $P\preceq I-P$, there is a partial isometry $V$ and a projection $Q$ in $\mathcal{M}$ such that $P=V^*V$ and $Q=VV^*\leqslant I-P$.
By \Cref{lem VBA}, we have
\begin{equation*}
  \Phi(BA)=\Phi(V^*VBA)=V^*\Phi(VBA)=V^*VB\Phi(A)=B\Phi(A).
\end{equation*}
This completes the proof.
\end{proof}

If $V$ is a partial isometry in $\mathscr M$ such that $V^2=0$, then $\Phi(V)=V\Phi(I)$ by \Cref{lem VBA}.
Note that $V^2=0$ implies that $VV^*\sim V^*V\leqslant I-VV^*$.
In the following proposition, we replace the condition $V^2=0$ with the weaker condition $VV^*\preceq I-VV^*$.

\begin{proposition}\label{prop W}
Let $\Phi$ be an ortho-additive range-contractive map on a von Neumann algebra $\mathscr{M}$.
If $W$ is a partial isometry in $\mathscr M$ such that $WW^*\preceq I-WW^*$, then $\Phi(W)=W\Phi(I)$.
\end{proposition}

\begin{proof}
Let $P=WW^*$.
Then $P\preceq I-P$.
According to the type decomposition theorem of von Neumann algebras, the proof is divided into three cases.

\textbf{Case I.}
$\mathscr M$ is a properly infinite von Neumann algebra.

Since $P\preceq I-P$ and the identity $I$ is a properly infinite projection, $I-P$ is also a properly infinite projection.
By the halving lemma, there are projections $P_1$ and $P_2$ in $\mathscr M$ such that
\begin{equation*}
  I-P=P_1+P_2\sim P_1\sim P_2.
\end{equation*}
Since $P\preceq I-P\sim P_1$, we have $I=P+(I-P)\preceq P_1+P_2=I-P$.
It follows that $I\sim(I-P)$ and hence $I\sim P_1\sim P_2$.
Since $I-W^*W\leqslant I\sim P_1$, there exists a partial isometry $W_1$ and a projection $Q$ in $\mathscr M$ such that
\begin{equation*}
  I-W^*W=W_1^*W_1\quad\text{and}\quad Q=W_1W_1^*\leqslant P_1\leqslant I-P=I-WW^*.
\end{equation*}
Let $U=W+W_1$ and $E=UU^*$.
A direct computation shows that
\begin{equation*}
  U^*U=I,\quad PU=W,\quad\text{and}\quad P\leqslant E=P+Q\leqslant P+P_1=I-P_2.
\end{equation*}
Hence $E\sim I$.
Since $I\sim P_2\leqslant I-E$, we have $E\sim(I-E)\sim I$.
Let
\begin{equation*}
  E_1=UEU^*\leqslant E\quad\text{and}\quad E_2=U(I-E)U^*=E-E_1.
\end{equation*}
Since $\mathcal{R}(E_1U)\vee\mathcal{R}(U^*E_1)\leqslant E$ and $\mathcal{R}(E_2U)\vee\mathcal{R}(U^*E_2)\leqslant I-E_1$, it is clear that
\begin{equation*}
  E_1U\in E\mathscr ME\quad\text{and}\quad E_2U\in (I-E_1)\mathscr M(I-E_1).
\end{equation*}
Since $E_1\sim E\sim I$ and $I\sim I-E\leqslant I-E_1$, we have $E_1\sim(I-E_1)\sim I$.
By the ortho-additivity of $\Phi$ and \Cref{lem BA}, we can get
\begin{equation*}
  \Phi(U)=\Phi(E_1U+E_2U)=\Phi(E_1U)+\Phi(E_2U)=E_1U\Phi(I)+E_2U\Phi(I)=U\Phi(I).
\end{equation*}
Therefore, $\Phi(W)=\Phi(PU)=P\Phi(U)=PU\Phi(I)=W\Phi(I)$ by \Cref{lem PA}.

\textbf{Case II.}
$\mathscr M$ is a type $\mathrm{II}_1$ von Neumann algebra.

By \cite[Lemma 6.5.6]{KR2}, there are equivalent orthogonal projections $\{P_j\}_{j=1}^4$ in $\mathscr M$ with sum $P$.
Recall that $P=WW^*$.
Let $E_j=P_j\vee W^*P_jW$ for every $1\leqslant j\leqslant 4$.
Since $W^*P_jW\sim P_j\sim P_i$ for all $1\leqslant i,j\leqslant 4$, by \cite[Exercise 6.9.3(i)]{KR2}, we can get
\begin{equation*}
  E_j\preceq P_1+P_2\sim P_3+P_4\leqslant I-(P_1+P_2).
\end{equation*}
Since $\mathscr M$ is a finite von Neumann algebra, we have $P_1+P_2\preceq I-E_j$.
Hence $E_j\preceq I-E_j$ for all $1\leqslant j\leqslant 4$.
It follows from $P_jW\in E_j\mathscr M E_j$ and \Cref{lem BA} that $\Phi(P_jW)=P_jW\Phi(I)$.
Combining with \Cref{lem PA}, we have
\begin{equation*}
  \Phi(W)=P\Phi(W)=\sum P_j\Phi(W)=\sum\Phi(P_jW)=\sum P_jW\Phi(I)=W\Phi(I).
\end{equation*}

\textbf{Case III.} $\mathscr M$ is a finite type $\mathrm{I}$ von Neumann algebra.

We claim that if $V$ is a partial isometry in $\mathscr M$ such that $VV^*\leqslant P$ and $VV^*$ is an abelian projection, then
\begin{equation}\label{equ abelian}
  \Phi(V)=V\Phi(I).
\end{equation}
Let $E=VV^*$, $F=V^*V$, $G=E\vee F$, and $G_0=E\wedge F$.
Then $V\in G\mathscr MG$.
Since $E$ is an abelian projection, $F$ is also abelian.
It follows that
\begin{equation*}
  V^*G_0V=FV^*FG_0FVF=FVFV^*FG_0F=FEFG_0=G_0.
\end{equation*}
Thus, $(E-G_0)V$ is a partial isometry with initial projection $F-G_0$ and final projection $E-G_0$.
In particular, we have $F-G_0\sim E-G_0$.
Since $E-G_0\sim G-F$ by Kaplansky's formula (see \cite[Theorem 6.1.7]{KR2}), there is a partial isometry $U$ in $\mathscr M$ such that
\begin{equation*}
  U^*U=F-G_0\quad\text{and}\quad UU^*=G-F.
\end{equation*}
Since $F\sim E\leqslant P\preceq I-P$ and $\mathscr M$ is a finite von Neumann algebra, we can get $F\preceq I-F$.
Note that $U^*VU^*U\in F\mathscr MF$.
It follows from \Cref{lem VBA} and \Cref{lem BA} that
\begin{equation*}
  \Phi(UU^*VU^*U)=U\Phi(U^*VU^*U)=UU^*VU^*U\Phi(I),
\end{equation*}
i.e., $\Phi((G-F)V(F-G_0))=(G-F)V(F-G_0)\Phi(I)$.
Since $(G-F)VG_0=0$, it is straightforward to verify that
\begin{equation*}
  V=GVF=FVF+(G-F)V(F-G_0).
\end{equation*}
Moreover, $\Phi(FVF)=FVF\Phi(I)$ by \Cref{lem BA}.
By the ortho-additivity of $\Phi$, we have
\begin{align*}
  \Phi(V) &=\Phi(FVF+(G-F)V(F-G_0))\\
  &=FVF\Phi(I)+(G-F)V(F-G_0)\Phi(I)=V\Phi(I).
\end{align*}
This completes the proof of \eqref{equ abelian}.

Since $\mathscr M$ is a finite type $\mathrm{I}$ von Neumann algebra, there exists a family of mutually orthogonal abelian projections $\{P_j\}_{j\in\Lambda}$ in $\mathscr M$ with sum $P$.
Then $\Phi(P_jW)=P_jW\Phi(I)$ by \eqref{equ abelian}.
It follows from \Cref{lem PA} that
\begin{equation*}
  \Phi(W)=P\Phi(W)=\sum P_j\Phi(W)=\sum\Phi(P_jW)=\sum P_jW\Phi(I)=W\Phi(I).
\end{equation*}

\textbf{General Case.}
By the type decomposition theorem, there are central projections $Z_1$, $Z_2$, and $Z_3$ in $\mathscr M$ with sum $I$ such that $\mathscr MZ_1$ is properly infinite or $Z_1=0$, $\mathscr MZ_2$ is of type $\mathrm{II}_1$ or $Z_2=0$, and $\mathscr MZ_3$ is a finite type $\mathrm{I}$ von Neumann algebra or $Z_3=0$.
By \Cref{lem PA}, we can define a map $\Phi_j\colon\mathscr MZ_j\to\mathscr MZ_j$ by
\begin{equation*}
  \Phi_j(AZ_j)=\Phi(AZ_j)\quad\text{for all}~AZ_j\in\mathscr MZ_j.
\end{equation*}
Then $\Phi_j$ is an ortho-additive range-contractive map on $\mathscr MZ_j$ for every $1\leqslant j\leqslant 3$.
It follows that
\begin{equation*}
  \Phi(Z_jW)=\Phi_j(Z_jW)=Z_jW\Phi_j(Z_j)=Z_jW\Phi(Z_j)=Z_jW\Phi(I)
\end{equation*}
and then
\begin{equation*}
  \Phi(W)=\sum Z_j\Phi(W)=\sum\Phi(Z_jW)=\sum Z_jW\Phi(I)=W\Phi(I).
\end{equation*}
We complete the proof.
\end{proof}

Now we are ready to prove the main result in this section.

\begin{theorem}\label{thm main1}
Suppose $\mathscr M$ is a von Neumann algebra without direct summand of type
$\mathrm{I}_1$ and $\Phi$ is an ortho-additive range-contractive map on $\mathscr{M}$.
Then $\Phi(A)=A\Phi(I)$ for every $A\in\mathscr M$.
\end{theorem}

\begin{proof}
Let $P$ be a projection in $\mathscr M$ such that $P\preceq I-P$.
By \cite[Proposition 6.1.6]{KR2}, there is a partial isometry in $\mathscr M$ such that
\begin{equation*}
  WW^*=\mathcal{R}(PA)\leqslant P\quad\text{and}\quad W^*W=\mathcal{R}((PA)^*).
\end{equation*}
It is clear that $WW^*\preceq I-WW^*$.
By \Cref{lem BA} and \Cref{prop W}, we have
\begin{equation*}
  \Phi(PA)=\Phi(PAW^*PW)=PAW^*P\Phi(W)=PAW^*PW\Phi(I)=PA\Phi(I).
\end{equation*}
Since $\mathscr M$ is a von Neumann algebra without direct summand of type $\mathrm{I}_1$, there are orthogonal projections $\{P_j\}_{j=1}^3$ in $\mathscr M$ with sum $I$ such that $P_j\preceq I-P_j$ for every $1\leqslant j\leqslant 3$.
More precisely, if $\mathscr M$ is properly infinite, or of type $\mathrm{II}_1$, or of type $\mathrm{I}_{2n}$, then we can assume that $P_1=0$ and $P_2\sim P_3$ by \cite[Lemma 6.3.3, Lemma 6.5.6]{KR2}.
If $\mathscr M$ is of type $\mathrm{I}_{2n+1}$, then we can assume that $P_1$ is abelian and $P_2\sim P_3$.
It follows from \Cref{lem PA} that
\begin{equation*}
  \Phi(A)=\sum P_j\Phi(A)=\sum\Phi(P_jA)=\sum P_jA\Phi(I)=A\Phi(I).
\end{equation*}
This completes the proof.
\end{proof}

\begin{remark}\label{rem main1}
If the map $\Phi$ is bijective in \Cref{thm main1}, then $\Phi(I)$ is invertible.
Indeed, there exists $A_0\in\mathscr M$ such that $\Phi(A_0)=I$, i.e., $A_0\Phi(I)=I$.
Let $B_0=\Phi(I)A_0$.
Then
\begin{equation*}
    \Phi(B_0-I)=(B_0-I)\Phi(I)=\Phi(I)A_0\Phi(I)-\Phi(I)=0.
\end{equation*}
It follows that $B_0-I=0$, i.e., $\Phi(I)A_0=I$.
Therefore, $\Phi(I)$ is invertible.
\end{remark}

Similar to \Cref{def ortho} (iv), a map $\Phi$ on $\mathscr M$ is called {\em range-expansive} if
\begin{equation*}
  \mathcal R(A)\leqslant\mathcal R(\Phi(A))\quad\text{for every}~A\in\mathscr M.
\end{equation*}
As a consequence of \Cref{thm main1}, we have the following result.

\begin{corollary}\label{cor expansive}
Suppose $\mathscr M$ is a von Neumann algebra without direct summand of type
$\mathrm{I}_1$ and $\Phi$ is an ortho-additive range-expansive bijective map on $\mathscr{M}$.
Then $\Phi(A)=A\Phi(I)$ for every $A\in\mathscr M$.
\end{corollary}

\begin{proof}
Let $\Psi$ be the inverse map of $\Phi$.
It is clear that $\Psi$ is range-contractive.
In particular, if $A$ and $B$ are orthogonal operators in $\mathscr M$, then $\Psi(A)$ and $\Psi(B)$ are orthogonal.
Since $\Phi$ is ortho-additive, we have
\begin{equation*}
  \Phi(\Psi(A)+\Psi(B))=\Phi(\Psi(A))+\Phi(\Psi(B))=A+B,
\end{equation*}
i.e., $\Psi(A+B)=\Psi(A)+\Psi(B)$.
Hence $\Psi$ is ortho-additive.
It follows from \Cref{thm main1} that $\Psi(A)=A\Psi(I)$ for every $A\in\mathscr M$.
Note that $\Psi(I)$ is invertible by \Cref{rem main1}, whose inverse is denoted by $T$.
Then $\Psi(A)T=A$ for every $A\in\mathscr M$.
Substituting $A$ for $\Phi(A)$, we have $\Phi(A)=AT$.
In particular, $\Phi(I)=T$.
This completes the proof.
\end{proof}

Note that the map $\Phi$ is not required to be bijective in \Cref{thm main1}.
We present the following question.

\begin{question}\label{que expansive}
Suppose $\mathscr M$ is a von Neumann algebra without direct summand of type
$\mathrm{I}_1$ and $\Phi$ is an ortho-additive range-expansive map on $\mathscr{M}$.
Do we have $\Phi(A)=A\Phi(I)$ for every $A\in\mathscr M$?
\end{question}

\section{Ortho-additive ortho-isomorphisms}\label{sec ortho-isomorphism}

In this section, we will characterize the structure of ortho-additive ortho-isomorphisms between von Neumann algebras (see \Cref{def ortho}).
Due to \Cref{eg strange}, we assume in \Cref{thm main2} that $\mathscr M$ has no direct summand of type $\mathrm{I}_1$.
The case of type $\mathrm{I}_2$ von Neumann algebras will be handled separately in \Cref{subsec type-I2}.

\subsection{Projection lattices}
Recall that $\mathscr P(\mathscr M)$ denotes the set of all projections in a von Neumann algebra $\mathscr M$.
The following lemma states that every ortho-isomorphism between von Neumann algebras naturally induces a projection ortho-isomorphism.

\begin{lemma}\label{lem theta}
Suppose $\mathscr M$ and $\mathscr N$ are von Neumann algebras, and $\varphi\colon\mathscr M\to\mathscr N$ is an ortho-isomorphism.
Then the map $\theta\colon\mathscr P(\mathscr M)\to\mathscr P(\mathscr N)$ defined by
\begin{equation*}
  \theta(\mathcal{R}(A))=\mathcal{R}(\varphi(A))\quad\text{for all}~A\in\mathscr M
\end{equation*}
is a projection ortho-isomorphism.
\end{lemma}

\begin{proof}
Let $A$ and $B$ be operators in $\mathscr M$ with $\mathcal{R}(A)\leqslant\mathcal{R}(B)$.
Since $\varphi$ is an ortho-isomorphism, for any operator $C$ in $\mathscr M$ with $\varphi(C)\perp\varphi(B)$, we have $C\perp B$.
It follows that $C\perp A$ and hence $\varphi(C)\perp\varphi(A)$.
Since $\varphi$ is bijective, we obtain that
\begin{equation*}
  \mathcal{R}(\varphi(A))\leqslant\mathcal{R}(\varphi(B)).
\end{equation*}
In particular, $\mathcal{R}(A)=\mathcal{R}(B)$ implies $\mathcal{R}(\varphi(A))=\mathcal{R}(\varphi(B))$.
Thus, the map $\theta$ is well-defined.
It is clear that $\theta$ is bijective and preserves orthogonality in both directions.
\end{proof}

In the next proposition, we characterize the structure of ortho-isomorphisms.

\begin{proposition}\label{prop ortho-isomorphism}
Suppose $\mathscr M$ and $\mathscr N$ are von Neumann algebras such that $\mathscr M$ has no direct summand of type $\mathrm{I}_2$, and $\varphi\colon\mathscr M\to\mathscr N$ is an ortho-isomorphism.
Then there exists a map $\Phi\colon\mathscr M\to\mathscr N$ such that
\begin{equation*}
  \Phi(\mathcal{R}(A))=\mathcal{R}(\Phi(A))=\mathcal{R}(\varphi(A))\quad\text{for all}~ A\in \mathscr M,
\end{equation*}
and $\Phi$ is the direct sum of a *-isomorphism and a conjugate *-isomorphism.
\end{proposition}

\begin{proof}
Let $\theta$ be the map given by \Cref{lem theta}.
By \Cref{thm Dye}, $\theta$ can be extended to a map $\Phi\colon\mathscr M\to\mathscr N$ which is the direct sum of a *-isomorphism and a conjugate *-isomorphism.
It follows that
\begin{equation*}
  \Phi(\mathcal{R}(A))=\theta(\mathcal{R}(A))=\mathcal{R}(\varphi(A))
  \quad\text{for all}~A\in\mathscr M.
\end{equation*}
Next, we need to show that $\Phi(\mathcal{R}(A))=\mathcal{R}(\Phi(A))$ for every $A\in\mathscr M$.
In fact, it follows from the identity $\mathcal{R}(A)A=A$ that $\Phi(\mathcal{R}(A))\Phi(A)=\Phi(A)$.
Since $\Phi(\mathcal{R}(A))$ is a projection, we have $\mathcal{R}(\Phi(A))\leqslant\Phi(\mathcal{R}(A))$.
Since $\Phi^{-1}$ is also the direct sum of a *-isomorphism and a conjugate *-isomorphism, we have $\mathcal{R}(\Phi^{-1}(A))\leqslant\Phi^{-1}(\mathcal{R}(A))$, i.e., $\Phi(\mathcal{R}(A))\leqslant\mathcal{R}(\Phi(A))$ for every $A\in\mathscr M$.
This completes the proof.
\end{proof}

Suppose $\mathscr M$ and $\mathscr N$ are finite von Neumann algebras, and $\varphi\colon\mathscr M\to\mathscr N$ is an ortho-isomorphism.
Since $0$ is the only operator orthogonal to all operators, we must have $\varphi(0)=0$.
For any operator $A\in\mathscr M$ with $A\perp\varphi(I)$, we have $\varphi^{-1}(A)\perp I$.
It follows that $\varphi^{-1}(A)=0$ and hence $A=\varphi(0)=0$.
Thus, $\mathcal R(\varphi(I))=I$.
Since $\mathscr N$ is a finite von Neumann algebra, we must have $\ker\varphi(I)=0$.
Recall that $\eta\mathscr N$ is the set of all (unbounded) operators affiliated with $\mathscr N$.
Therefore, the inverse $\varphi(I)^{-1}$ of $\varphi(I)$ is well-defined in $\eta\mathscr N$.
Since $\mathscr N$ is a finite von Neumann algebra, $\eta\mathscr N$ is a *-algebra.
We define a map $\Phi\colon\mathscr M\to\eta\mathscr N$ by
\begin{equation*}
  \Phi(A)=\varphi(A)\varphi(I)^{-1}.
\end{equation*}
Then $\Phi$ is an ortho-preserving map from $\mathscr M$ into $\eta\mathscr N$.
The following lemma also provides a projection ortho-isomorphism, and the method differs from that of \Cref{lem theta}.

\begin{lemma}\label{lem Phi-on-P}
Suppose $\mathscr M$ and $\mathscr N$ are finite von Neumann algebras, and $\varphi\colon\mathscr M\to\mathscr N$ is an ortho-additive ortho-isomorphism.
Let $\Phi\colon\mathscr M\to\eta\mathscr N$ be the map defined by
\begin{equation*}
  \Phi(A)=\varphi(A)\varphi(I)^{-1}\quad\text{for all}~A\in\mathscr M.
\end{equation*}
Then the restriction $\Phi|\mathscr P(\mathscr M)$ is an ortho-isomorphism from $\mathscr P(\mathscr M)$ onto $\mathscr P(\mathscr N)$.
\end{lemma}

\begin{proof}
Based on the discussion preceding this lemma, the map $\Phi$ is well-defined.
Since $\varphi$ is ortho-additive and ortho-preserving, for any projection $P$ in $\mathscr M$, we have
\begin{equation*}
  \Phi(P)+\Phi(I-P)=I\quad\text{and}\quad\Phi(P)^*\Phi(I-P)=0.
\end{equation*}
It follows that $\Phi(P)^*\Phi(P)=\Phi(P)^*$.
Hence $\Phi(P)$ is a projection in $\mathscr N$.
Let $Q$ be a projection in $\mathscr N$.
Since $\varphi$ is onto, there exists $A\in\mathscr M$ such that $\varphi(A)=Q\varphi(I)$.
Let $P$ be the range projection of $A$.
Then $\mathcal R(\varphi(P))=\mathcal R(\varphi(A))$ by \Cref{lem theta}.
Since $\Phi(P)$ is a projection, we can get
\begin{equation*}
  \Phi(P)=\mathcal R(\Phi(P))=\mathcal R(\varphi(P))=\mathcal R(\varphi(A))=Q.
\end{equation*}
Therefore, $\Phi$ maps $\mathscr P(\mathscr M)$ onto $\mathscr P(\mathscr N)$.
This completes the proof.
\end{proof}

\subsection{Type $\mathrm{I}_2$ von Neumann algebras}\label{subsec type-I2}
In this subsection, we will establish the structure of ortho-additive ortho-isomorphisms between type $\mathrm{I}_2$ von Neumann algebras.
Suppose $\mathscr M$ and $\mathscr N$ are type $\mathrm{I}_2$ von Neumann algebras.
By \cite[Theorem 6.6.5]{KR2}, we may assume that $\mathscr M=M_2(\mathscr A)$ and $\mathscr N=M_2(\mathscr B)$, where $\mathscr A$ and $\mathscr B$ are abelian von Neumann algebras.
Then every operator in $\mathscr M$ can be written as an operator matrix
$
\begin{pmatrix}
  a_{11} & a_{12} \\
  a_{21} & a_{22}
\end{pmatrix}$, where $a_{ij}\in\mathscr A$.
Let
\begin{equation}\label{equ E-ij}
  E_{11}=
  \begin{pmatrix}
    1 & 0 \\
    0 & 0
  \end{pmatrix},
  E_{12}=
  \begin{pmatrix}
    0 & 1 \\
    0 & 0
  \end{pmatrix},
  E_{21}=
  \begin{pmatrix}
    0 & 0 \\
    1 & 0
  \end{pmatrix},~\text{and}~
  E_{22}=
  \begin{pmatrix}
    0 & 0 \\
    0 & 1
  \end{pmatrix}.
\end{equation}
Then $\{E_{ij}\}_{1\leqslant i,j\leqslant 2}$ is a system of matrix units in $\mathscr M$ with $E_{11}+E_{22}=I$.
Recall that $\eta\mathscr N$ is the *-algebra of all (unbounded) operator affiliated with $\mathscr N$ and $\eta\mathscr N=M_2(\eta\mathscr B)$.

\begin{lemma}\label{lem Phi-E-F}
Suppose $\mathscr M$ and $\mathscr N$ are type $\mathrm{I}_2$ von Neumann algebras, and $\Phi\colon\mathscr M\to\eta\mathscr N$ is an ortho-additive ortho-preserving map.
Let $\{E_{ij}\}_{1\leqslant i,j\leqslant 2}$ and $\{F_{ij}\}_{1\leqslant i,j\leqslant 2}$ be systems of matrix units in $\mathscr M$ and $\mathscr N$, respectively.
If $\Phi(E_{jj})=F_{jj}$ for $j=1,2$, then $\Phi$ is a map from $\mathscr M$ into $\mathscr N$ and is the direct sum of a *-homomorphism and a conjugate *-homomorphism.
\end{lemma}

\begin{proof}
We may assume that $\mathscr M=M_2(\mathscr A)$ and $\mathscr N=M_2(\mathscr B)$, where $\mathscr A$ and $\mathscr B$ are abelian von Neumann algebras.
Without loss of generality, we assume that both $\{E_{ij}\}_{1\leqslant i,j\leqslant 2}$ and $\{F_{ij}\}_{1\leqslant i,j\leqslant 2}$ are of the form given by \eqref{equ E-ij}.
Since $\Phi(E_{22})=F_{22}$ and $\Phi$ is ortho-preserving, for any $a,b\in\mathscr A$, $E_{22}\perp
\begin{pmatrix}
  a & b \\
  0 & 0
\end{pmatrix}$ implies that
$F_{22}\perp\Phi(\begin{pmatrix}
  a & b \\
  0 & 0
\end{pmatrix})$.
Thus, there exist two maps $\Phi_1,\Phi_2\colon\mathscr A\times\mathscr A\to\eta\mathscr B$ such that
\begin{equation*}
  \Phi(\begin{pmatrix}
    a & b \\
    0 & 0
  \end{pmatrix})=\begin{pmatrix}
    \Phi_1(a,b) & \Phi_2(a,b) \\
    0 & 0
  \end{pmatrix}.
\end{equation*}
Similarly, there are two maps $\Phi_3,\Phi_4\colon\mathscr A\times\mathscr A\to\eta\mathscr B$ such that
\begin{equation*}
  \Phi(\begin{pmatrix}
    0 & 0 \\
    c & d
  \end{pmatrix})=\begin{pmatrix}
    0 & 0 \\
    \Phi_3(c,d) & \Phi_4(c,d)
  \end{pmatrix}.
\end{equation*}
Since $\begin{pmatrix}
  a & b \\
  0 & 0
\end{pmatrix}\perp\begin{pmatrix}
  0 & 0 \\
  c & d
\end{pmatrix}$, it follows from the ortho-additivity of $\Phi$ that
\begin{equation}\label{equ Phi-1234}
  \Phi(\begin{pmatrix}
    a & b \\
    c & d
  \end{pmatrix})=\begin{pmatrix}
    \Phi_1(a,b) & \Phi_2(a,b) \\
    \Phi_3(c,d) & \Phi_4(c,d)
  \end{pmatrix}.
\end{equation}
Since
$\begin{pmatrix}
  c & 0 \\
  -1 & 0
\end{pmatrix}
\perp
\begin{pmatrix}
  a & b \\
  c^*a & c^*b
\end{pmatrix}$ and $\Phi$ is an ortho-additive ortho-preserving map, we have
\begin{equation*}
  \Phi(\begin{pmatrix}
    a+c & b \\
    c^*a-1 & c^*b
  \end{pmatrix})=\Phi(\begin{pmatrix}
    a & b \\
    c^*a & c^*b
  \end{pmatrix})+\Phi(\begin{pmatrix}
    c & 0 \\
    -1 & 0
  \end{pmatrix})\quad\text{and}\quad\Phi(\begin{pmatrix}
    c & 0 \\
    -1 & 0
  \end{pmatrix})\perp\Phi(\begin{pmatrix}
    a & b \\
    c^*a & c^*b
  \end{pmatrix}).
\end{equation*}
Combining with \eqref{equ Phi-1234}, we can get
\begin{equation}\label{equ Phi1-additive}
  \Phi_1(a+c,b)=\Phi_1(a,b)+\Phi_1(c,0)
\end{equation}
and
\begin{equation}\label{equ Phi1-ortho}
  \Phi_1(c,0)^*\Phi_1(a,b)+\Phi_3(-1,0)^*\Phi_3(c^*a,c^*b)=0.
\end{equation}
Since $\Phi(E_{11})=F_{11}$, we have $\Phi_1(1,0)=1$ and $\Phi_1(0,1)=0$.
It follows from $\Phi_1(1,0)=1$ and \eqref{equ Phi1-ortho} that
\begin{equation}\label{equ Phi1-multiplicative}
  \Phi_1(c^*a,c^*b)=\Phi_1(c,0)^*\Phi_1(a,b).
\end{equation}
In particular, $\Phi_1(0,c^*)=\Phi_1(c,0)^*\Phi_1(0,1)=0$, i.e., $\Phi_1(0,b)=0$ for every $b\in\mathscr A$.
By \eqref{equ Phi1-additive}, we can get $\Phi_1(c,b)=\Phi_1(0,b)+\Phi_1(c,0)=\Phi_1(c,0)$.
Thus, we have
\begin{equation}\label{equ ab=a0}
  \Phi_1(a,b)=\Phi_1(a,0)\quad\text{for all}~a,b\in\mathscr A.
\end{equation}
Next, we define a map $f\colon\mathscr A\to\eta\mathscr B$ by $f(a)=\Phi_1(a,0)$.
Then $f(1)=1$, $f(a+b)=f(a)+f(b)$ by \eqref{equ Phi1-additive}, and $f(c^*a)=f(c)^*f(a)$ by \eqref{equ Phi1-multiplicative}.
Therefore, $f$ is a unital *-ring homomorphism from $\mathscr A$ into $\eta\mathscr B$.
By \Cref{lem *ring-homomorphism}, $f$ is a map from $\mathscr A$ into $\mathscr B$ and is the direct sum of a *-homomorphism and a conjugate *-homomorphism.
Moreover, it follows from \eqref{equ ab=a0} that
\begin{equation}\label{equ Phi1=f}
  \Phi_1(a,b)=f(a)\quad\text{for all}~a,b\in\mathscr A.
\end{equation}
Taking $a=-1$, $b=0$, and $c=1$ in \eqref{equ Phi1-ortho}, we can get $|\Phi_3(-1,0)|^2=1$.
Hence $\Phi_3(-1,0)$ is a unitary operator in $\mathscr B$.
Let $u=-\Phi_3(1,0)$.
By \eqref{equ Phi1-ortho}, we obtain that
\begin{equation}\label{equ Phi3=f}
  \Phi_3(a,b)=uf(a)\quad\text{for all}~a,b\in\mathscr A.
\end{equation}
Similarly, there exists a unital *-ring homomorphism $g\colon\mathscr A\to\eta\mathscr B$ and a unitary operator $v$ in $\mathscr B$ such that
\begin{equation*}
  \Phi_4(a,b)=g(b),\quad\Phi_2(a,b)=vg(b)\quad\text{for all}~a,b\in\mathscr A.
\end{equation*}
It follows from \eqref{equ Phi-1234} that
\begin{equation}\label{equ Phi=fg}
  \Phi(\begin{pmatrix}
    a & b \\
    c & d
  \end{pmatrix})=\begin{pmatrix}
    f(a) & vg(b) \\
    uf(c) & g(d)
  \end{pmatrix}.
\end{equation}
Since
$\begin{pmatrix}
  a^* & 0 \\
  -1 & 0
\end{pmatrix}
\perp
\begin{pmatrix}
  0 & 1 \\
  0 & a
\end{pmatrix}$ and $\Phi$ is ortho-preserving, we can get
\begin{equation*}
  vf(a)=u^*g(a).
\end{equation*}
Thus, $v=u^*$ and $f=g$.
Let $\Phi_f=I_2\otimes f$ and
$U=\begin{pmatrix}
  1 & 0 \\
  0 & u
\end{pmatrix}$.
Then $\Phi_f$ is the direct sum of a *-homomorphism and a conjugate *-homomorphism.
Moreover, we can rewrite \eqref{equ Phi=fg} as
\begin{equation*}
  \Phi(\begin{pmatrix}
    a & b \\
    c & d
  \end{pmatrix})=U\Phi_f(\begin{pmatrix}
    a & b \\
    c & d
  \end{pmatrix})U^*.
\end{equation*}
This completes the proof.
\end{proof}

\begin{proposition}\label{prop type-I2}
Suppose $\mathscr M$ and $\mathscr N$ are type $\mathrm{I}_2$ von Neumann algebras, and $\varphi\colon\mathscr M\to\mathscr N$ is an ortho-additive ortho-isomorphism.
Then there exists a map $\Phi\colon\mathscr M\to\mathscr N$ such that
\begin{equation*}
  \varphi(A)=\Phi(A)\varphi(I)\quad\text{for all}~A\in\mathscr M,
\end{equation*}
and $\Phi$ is the direct sum of a *-isomorphism and a conjugate *-isomorphism.
\end{proposition}

\begin{proof}
Let $\Phi$ be the map given by \Cref{lem Phi-on-P} and $\{E_{ij}\}_{1\leqslant i,j\leqslant 2}$ a system of matrix units in $\mathscr M$ with $E_{11}+E_{22}=I$.
By \Cref{cor Dye} and \Cref{lem Phi-on-P}, $\Phi(E_{11})$ and $\Phi(E_{22})$ are abelian projections in $\mathscr N$ with sum $I$.
Thus, there exists a system of matrix units $\{F_{ij}\}_{1\leqslant i,j\leqslant 2}$ in $\mathscr N$ such that $F_{jj}=\Phi(E_{jj})$ for $j=1,2$.
Then $\Phi$ is a map from $\mathscr M$ into $\mathscr N$ and is the direct sum of a *-homomorphism and a conjugate *-homomorphism by \Cref{lem Phi-E-F}.
Since $\varphi$ is injective, $\Phi$ is also injective.
Since $\varphi$ is surjective, for any $B\in\mathscr N$, there exists $A\in\mathscr M$ such that $\varphi(A)=B\varphi(I)$, i.e., $\Phi(A)=B$.
It follows that $\Phi$ is surjective.
This completes the proof.
\end{proof}

\subsection{Ortho-additive ortho-isomorphisms}\label{subsec ortho-isomorphism}
For type $\mathrm{I}_2$ von Neumann algebras, we have obtained the structure of ortho-additive ortho-isomorphisms in \Cref{prop type-I2}.
By \Cref{eg strange}, we cannot expect to have a similar result for type $\mathrm{I}_1$ von Neumann algebras, i.e., abelian von Neumann algebras.
In the next proposition, we consider von Neumann algebras without direct summands of type $\mathrm{I}_1$ and type $\mathrm{I}_2$.

\begin{proposition}\label{prop no-type-I2}
Suppose $\mathscr M$ and $\mathscr N$ are von Neumann algebras such that $\mathscr M$ has no direct summands of type $\mathrm{I}_1$ and type $\mathrm{I}_2$, and $\varphi\colon\mathscr M\to\mathscr N$ is an ortho-additive ortho-isomorphism.
Then there exists a map $\Phi\colon\mathscr M\to\mathscr N$ such that
\begin{equation*}
  \varphi(A)=\Phi(A)\varphi(I)\quad\text{for all}~A\in\mathscr M,
\end{equation*}
and $\Phi$ is the direct sum of a *-isomorphism and a conjugate *-isomorphism.
\end{proposition}

\begin{proof}
Let $\Phi\colon\mathscr M\to\mathscr N$ be the map given by \Cref{prop ortho-isomorphism} and we define a map $\psi$ on $\mathscr M$ by
\begin{equation*}
  \psi=\Phi^{-1}\circ\varphi\colon\mathscr M\to\mathscr M.
\end{equation*}
It is clear that $\psi$ is ortho-additive.
Furthermore, for any $A\in\mathscr M$, we have
\begin{equation*}
  \Phi(\mathcal{R}(\psi(A)))=\mathcal{R}(\Phi\circ\psi(A))
  =\mathcal{R}(\varphi(A))=\Phi(\mathcal{R}(A)),
\end{equation*}
i.e., $\mathcal{R}(\psi(A))=\mathcal{R}(A)$.
Thus, $\psi$ is range-contractive.
By \Cref{thm main1}, for any $A\in\mathscr M$, we have
\begin{equation*}
  \Phi^{-1}\circ\varphi(A)=\psi(A)=A\psi(I).
\end{equation*}
It follows that $\varphi(A)=\Phi(A)\varphi(I)$.
This completes the proof.
\end{proof}

The next theorem is the main result in this section.

\begin{theorem}\label{thm main2}
Suppose $\mathscr M$ and $\mathscr N$ are von Neumann algebras such that $\mathscr M$ has no direct summand of type $\mathrm{I}_1$, and $\varphi\colon\mathscr M\to\mathscr N$ is an ortho-additive ortho-isomorphism.
Then there exists a map $\Phi\colon\mathscr M\to\mathscr N$ such that
\begin{equation*}
  \varphi(A)=\Phi(A)\varphi(I)\quad\text{for all}~A\in\mathscr M,
\end{equation*}
and $\Phi$ is the direct sum of a *-isomorphism and a conjugate *-isomorphism.
\end{theorem}

\begin{proof}
By the type decomposition theorem, there exists a maximal central projection $Z_2$ in $\mathscr M$ such that $\mathscr MZ_2$ is of type $\mathrm{I}_2$ or $Z_2=0$.
Let $Z_1=I-Z_2$.
Then $\mathscr MZ_1$ has no direct summands of type $\mathrm{I}_1$ and type $\mathrm{I}_2$ by assumption.
Let $\theta$ be the map given by \Cref{lem theta}.
By \Cref{cor Dye}, $\theta(Z_2)$ is the maximal central projection in $\mathscr N$ such that $\mathscr N\theta(Z_2)$ is of type $\mathrm{I}_2$ or $\theta(Z_2)=0$.
For $j=1,2$, let $\mathscr M_j=\mathscr MZ_j$ and $\mathscr N_j=\mathscr N\theta(Z_j)$.
Then
\begin{equation*}
  \mathscr M=\mathscr M_1\oplus\mathscr M_2\quad\text{and}\quad
  \mathscr N=\mathscr N_1\oplus\mathscr N_2.
\end{equation*}
By \Cref{lem theta}, we have $\varphi(\mathscr M_j)=\mathscr N_j$.
Let $\varphi_j\colon\mathscr M_j\to\mathscr N_j$ be the restriction of $\varphi$ on $\mathscr M_j$.
Then $\varphi_j$ is an ortho-additive ortho-isomorphism.
By \Cref{prop type-I2} and \Cref{prop no-type-I2}, there exists a map $\Phi_j\colon\mathscr M_j\to\mathscr N_j$ such that
\begin{equation*}
  \varphi_j(A)=\Phi_j(A)\varphi_j(Z_j)\quad\text{for all}~A\in\mathscr M_j,
\end{equation*}
and $\Phi_j$ is the direct sum of a *-isomorphism and a conjugate *-isomorphism.
Let $\Phi=\Phi_1\oplus\Phi_2$.
This completes the proof.
\end{proof}

\begin{remark}\label{rem main2}
Similar to \Cref{rem main1}, it can be shown that $\varphi(I)$ is invertible in \Cref{thm main2}.
Actually, there exists $A_0\in\mathscr M$ such that $\varphi(A_0)=I$, i.e., $\Phi(A_0)\varphi(I)=I$.
Let $B_0=\varphi(I)\Phi(A_0)$.
Then
\begin{equation*}
  (\varphi\circ\Phi^{-1})(B_0-I)=(B_0-I)\varphi(I)
  =\varphi(I)\Phi(A_0)\varphi(I)-\varphi(I)=0.
\end{equation*}
It follows that $B_0-I=0$, i.e., $\varphi(I)\Phi(A_0)=I$.
Therefore, $\varphi(I)$ is invertible.
\end{remark}


\end{document}